\documentclass[12pt]{amsart}
\usepackage{amsmath, amsthm, amsfonts, amssymb, diagrams, times, 
txfonts, bbm, mathrsfs, array, longtable, verbatim, enumerate}
\usepackage{graphicx, color}
\usepackage[initials]{amsrefs}
\DeclareGraphicsExtensions{.pdf}

\title{A Localization in MV-algebras}
\author{Colin G. Bailey}
\address{School of Mathematics, Statistics \& Operations Research\\
Victoria University of Wellington\\
Wellington, New Zealand\\
}
\email{Colin.Bailey@vuw.ac.nz}

\date{\today}

\def\leftGen{{[\kern-1.5pt[}}
\def\rightGen{{]\kern-1.5pt]}}

\let\rsf\mathscr

\providecommand{\meet}{\mathbin{\wedge}}
\providecommand{\join}{\mathbin{\vee}}
\newcommand{\comp}[1]{\overline{#1}}

\ifx\upharpoonright\undefined
     \def\restrict{\hbox{\rm\kern0.166em\accent"12\kern-0.536em$\vert$\kern0.3em}}%
  \else
     \def\restrict{\upharpoonright}%
  \fi
\makeatletter
\def\twoSet#1#2{\left\{%
\vphantom{#2}#1\thinspace\right|\nolinebreak[3]\left.%
  #2%
  \vphantom{#1}%
  \right\}%
}
\def\oneSet#1{\left\lbrace#1\right\rbrace}

\newif\if@nstr
\def\setstrfalse{\let\if@nstr=\iffalse}
\def\setstrtrue{\let\if@nstr=\iftrue}
\def\@nstr #1#2{
\def\@@nstr ##1#1##2##3\@@nstr{\ifx
\@nstr ##2\setstrfalse \else \setstrtrue \fi }
\@@nstr #2#1\@nstr \@@nstr}
\def\@separate#1|#2@{\setFront{#1}\setBack{#2}}
\def\lb#1\rb{\@nstr|{#1} \if@nstr \@separate#1 @ \twoSet{\@setFront}{\@setBack}%
\else \@separate |{#1 }@ \oneSet{\@setBack}\fi%
}
\def\setFront#1{\def\@setFront{#1}}
\def\setBack#1{\def\@setBack{#1}}
\def\Set#1{\lb{#1}\rb}

\def\oneBrk#1{\left\langle#1\right\rangle}
\def\twoBrk#1#2{\left\langle%
\vphantom{#2}#1\thinspace\right|\nolinebreak[3]\left.%
  #2%
  \vphantom{#1}%
  \right\rangle%
}
\def\brk<#1>{\@nstr|{#1} \if@nstr \@separate#1 @ \twoBrk{\@setFront}{\@setBack}%
\else \@separate |{#1 }@ \oneBrk{\@setBack}\fi%
}

\def\lemref#1{\normalfont{lemma}~\ref{#1}}

\def\propref#1{\normalfont{proposition}~\ref{#1}}
\theoremstyle{plain}
\newtheorem{thm}{Theorem}[section]
\newtheorem{lem}[thm]{Lemma}
\newtheorem{cor}[thm]{Corollary}
\newtheorem{prop}[thm]{Proposition}
\newtheorem{defn}[thm]{Definition}

\theoremstyle{remark}
{}
{}
{}
{}

\@ifpackageloaded{bbm}{

\newcommand{\N}{{\mathbbm{N}}}

}{

\newcommand{\N}{{\mathbb{N}}}

}
\makeatother

\begin{document}
\begin{abstract} 	
    In this document we consider a way of localizing an MV-algebra. 
    Given any prime filter $F$ we find a local MV-algebra which has 
    the same poset of prime filters as the poset of prime filters 
    comparable to $F$.
\end{abstract}
\maketitle

\section{Introduction}
A local MV-algebra is one with a single maximal implication filter. 
Such algebras are of interest in the representation theory of 
MV-algebras (see \cite{LocalMVAlg} for example). 

The set of prime implication filters of an MV-algebra forms a spectral root system, 
ordered by set-inclusion. The existence of a unique maximal filter is 
equivalent to the stem of this root system  being nonempty. (The stem is the set 
$\text{Stem}=\Set{ P | P\text{ is a prime filter comparable to every 
other prime filter}}$.) Whenever the stem is non-empty 
it has a least element, the Conrad filter (defined below). This 
filter can be characterized in several ways, as we show in section 2 
below. This work
is heavily based on work of Conrad on 
lattice-ordered groups (see \cite{ConradLattice}), recasting his material in terms of implication filters in 
MV-algebras.  

In the last section we consider how to invert this characterization 
to get a prime filter into the stem of an MV-algebra. This 
localization takes a prime implication filter $P$ and finds a quotient 
in which the maximal filter over $P$ is the unique maximal filter, 
and the prime filter structure of the quotient is isomorphic to the 
set of prime filters comparable to $P$. 

In most of what follows the filters are taken to be implication 
filters rather than lattice filters. We recall that an implication 
filter is a lattice filter closed under powers. 

Given an MV-algebra $\mathcal L$, there are several sets of filters that we are interested in:
\begin{align*}
    \text{PSpec}&=\Set{P | P\text{ is a prime implication filter of 
    }\mathcal L}\\
    &=\text{ the prime spectrum;}\\
    \text{PSpec}(F)&=\Set{P | P\text{ is a prime implication filter of 
    }\mathcal L\text{ comparable to }F};\\
    \mu\text{S}&=\Set{P | P\text{ is a minimal prime filter of 
    }\mathcal L}\\
    &=\text{the minimal spectrum;}\\
    \mu\text{S}(F)&=\Set{P | P\text{ is a minimal prime filter of 
    }\mathcal L\text{ comparable to }F}.
\end{align*}

Our notation usually follows that of \cite{MainBk} with the exception 
that we use $\otimes$ instead of $\odot$. 

\section{Counits}

\begin{defn}
    $u\in\mathcal L$ is a \emph{counit} iff $u<1$ and there exists 
    some $v<1$ with $u\join v=1$.
\end{defn}

\begin{defn}\label{def:ConradFilter}
    The \emph{Conrad filter} of an MV-algebra is the implication 
    filter generated by the counits.
    
    We usually denote it by $\rsf N(\mathcal L)$ or $N$.
\end{defn}

If $N=\rsf N(\mathcal L)$ then $N$ is 
prime as $a\join b=1$, $a,b<1$ implies $a$ and $b$ are counits and 
so in $N$.  

All implication filters that contain $N$ form a chain. The following 
lemma provides an alternative characterization of the prime filters 
in this chain. 

\begin{lem}
    Let $P$ be a prime implication filter. Then $P$ contains all counits 
    iff for all $x\notin P$ and all $p\in P$ $p\geq x$.
\end{lem}
\begin{proof}
    Suppose that $x\notin P$ and $y\in P$ with $x\not\le y$. 
    We know that $(x\to y)\join(y\to x)=1$. 
    
    As $x\not\le y$ we have $x\to y<1$, and $y\not\le x$ implies 
    $y\to x<1$ and so $y\to x$ is a co-unit. 
    
    If it is in $P$ then so is $x\meet y= (y\to x)\otimes y$, contradicting $x\notin P$. 
    Thus $P$ cannot contain 
    all co-units. 
    
    Conversely if $a$ is a co-unit and $a\join b=1$ for some $b>0$. 
    One of $a$ or $b$ is in $P$ (as $P$ is prime). If $a\notin P$ 
    then $a\le b$ which is impossible, so $a\in P$.  
\end{proof}

A slight variation of this proof lets us see that filters are incomparable 
because of counits.

\begin{lem}\label{lem:inCompFilCOU}
    Let $P$ and $Q$ be incomparable implication filters. Then there 
    is a counit in $Q\setminus P$.
\end{lem}
\begin{proof}
    Suppose not,  ie every counit in $Q$ is also in $P$. As $P$ and 
    $Q$ are incomparable we can find $x\in Q\setminus P$ and $y\in 
    P\setminus Q$. Thus $x\not\le y$ and $y\not\le x$ and so 
    $x\to y<1$ and $y\to x<1$ and $(x\to y)\join(y\to x)=1$. So
    $y\to x$ is a counit in $Q$ and (by assumption) must be in $P$. 
    As $y\in P$ we now have 
    $x\meet y= y\otimes(y\to x)\in P$ contradicting $x\notin P$.
\end{proof}

The next two results show that $N$ is actually the minimum prime 
filter comparable to all prime filters.

\begin{prop}\label{prop:incompUn}
    Let $P$ be any prime implication filter that does not contain all 
    counits. Then there is a prime implication filter incomparable 
    to $P$.
\end{prop}
\begin{proof}
    As $P$ does not contain all counits we know that there is some 
    $g\notin P$ that is not below $P$, ie there is some $p\in P$ 
    with $p\not\geq g$. Of course $g\not\geq p$. Thus 
    $g\to p<1$ and $p\to g<1$ and $(g\to p)\join(p\to
    g)=1$. 
    
    As $(p\to g)\otimes(p\join g)=g$ we must have $p\to g\notin P$. 
    
    Let $Q$ be maximal avoiding $g\to p$. Then $Q$ is prime and as $(g\to p)\join(p\to 
    g)=1\in Q$ we have $p\to g\in Q\setminus P$. 
    By construction $g\to p\in P\setminus Q$ and so these two 
    ideals are incomparable.
\end{proof}

\begin{prop}
    If $P$ is a prime implication filter then either $N\subseteq P$ or 
    $P\subseteq N$.
\end{prop}
\begin{proof}
    If $P$ is not a subset of $N$ then we can find $p\in P\setminus 
    N$. $p\notin N$ implies $p$ is below $N$ and so 
    $N\subseteq[p,1]\subseteq P$. 
\end{proof}

Thus $N$ is the minimum  prime implication filter comparable to all 
others. The existence of such a filter implies that $N$ is a proper 
filter,  as if we have a minimal prime implication filter $F$ comparable to 
all others then it must contain all counits -- by 
\propref{prop:incompUn} and so $N$ exists and so $F$ equals $N$. 

Since any desired root system is the root 
system of an MV-algebra (\cite{SpecRootSys}), we see that it is possible to have 
non-trivial $N$. 

\begin{prop}\label{prop:nonMinN}
    $N$ is a minimal prime implication filter iff $N=\Set1$.
\end{prop}
\begin{proof}
    The right to left implication is immediate.
    
    If $N$ is minimal then it is the unique minimal implication 
    filter and so must equal $\Set1$ -- as we know the intersection 
    of all minimal implication filters is $\Set1$.
    
    Or just notice that $\mathcal L$ embeds into $\prod_{m\in\mu 
    S}\mathcal L/m= \mathcal L/N$ is linearly ordered,  and so 
    $\mathcal L$ is linearly ordered which implies $N=\Set1$.
\end{proof}

We also note that if $N$ is proper then there is a 
unique maximal implication filter -- the one that contains $N$. We 
also have the converse.
\begin{prop}
    If there is exactly one maximal proper implication filter then it 
    contains all counits.
\end{prop}
\begin{proof}
    Let $M$ be the maximum implication filter. Let $a,b<1$ with 
    $a\join b=1$. Let 
    $F_{b}=\Set{ x| x\join b=1}$. Then $0\notin F_{b}$, $a\in F_{b}$ 
    and it is easy to see that $F_{b}$ is a lattice filter. Also, 
    $x\in F_{b}$ implies 
    $x^{n}\join b\geq x^{n}\join b^{n}= (x\join b)^{n}= 1$ and so 
    $F_{b}$ is an implication filter. Hence $a\in F_{b}\subseteq M$. 
\end{proof}

Thus, if there is a maximum implication filter $M$ then  $N\subseteq 
M$ and $N$ is proper.

\section{Localization}
Let $P$ be a prime implication filter. We seek a quotient of 
$\mathcal L$ in which $P$ contains the Conrad filter. The 
construction we give below also preserves the structure of 
$\text{PSpec}(P)$.

\begin{defn}
    Let $P$ be a prime implication filter. Then
$$
\ell(P)=\leftGen\Set{x\to p | p\in P\text{ and }x\notin P}\rightGen.
$$
\end{defn}

Because of \lemref{lem:inCompFilCOU} we need to quotient out by at 
least $\ell(P)$ in order to make $P$ contain all counits in a 
quotient.

It is clear that $\ell(P)\subseteq P$ as $x\to p\geq p$ for any $p\in 
P$.
In general this inclusion is strict,  with the only exception being minimal prime 
filters.
\begin{lem}\label{lem:minimalL}
    $P$ is minimal prime iff $\ell(P)=P$.
\end{lem}
\begin{proof}
    If $P$ is minimal prime and $p\in P$ then there is some $t\notin 
    P$ with $t\join p=1$. Therefore $t\to p= 1\to p= p\in\ell(P)$.
    
    If $\ell(P)=P$ and $p\in P$ then $p\geq x\to p'$ for some 
    $x\notin P$ and $p'\in P$. Now $p'\to x\notin P$ else we would 
    have $p'\otimes(p'\to x)= p'\meet x\in P$ and so $x\in P$. 
    Also $p\join(p'\to x)\geq (x\to p')\join(p'\to x)= 1$. Thus $P$ 
    must be minimal.
\end{proof}

The next few lemmas show the relationship of $\ell(P)$ to the minimal 
filters below $P$.

\begin{lem}\label{lem:minimalInt}
    If $m\subseteq P$ is minimal prime then $\ell(P)\subseteq m$.
\end{lem}
\begin{proof}
    Let $x\notin P$ and $p\in P$. Then $p\otimes(p\to x)= p\meet x$ 
    implies $p\to x\notin P$ and so is not in $m$. But
    $(x\to p)\join(p\to x)=1\in m$ and $m$ is prime,  so $x\to p\in 
    m$.
\end{proof}

\begin{lem}\label{lem:minimalConv}
    Let $p\in P\setminus\ell(P)$. Then there is some minimal prime filter
    $m\subseteq P$ with $p\notin m$.
\end{lem}
\begin{proof}        
    Look in $\mathcal L/\ell(P)$. Then $[p]\not=1$ and is in $P/\ell(P)$.  
    We also know that the Conrad filter of $\mathcal L/\ell(P)$ is 
    contained in $P/\ell(P)$ -- since $x\notin P$ and $p\in P$ implies 
    $x\to p\in\ell(P)$ and so $x\le p\mod\ell(P)$. All minimal filters 
    must be subsets of the Conrad filter and so take $M$ to be a minimal 
    prime filter of $\mathcal L/\ell(P)$ that avoids $[p]<[1]$. Then 
    $M\subseteq P/\ell(P)$ and so the preimage $M'$ gives a prime subfilter of 
    $P$ that avoids $p$. 
    
    Any minimal filter of $\mathcal L$ contained in $M'$ works.
\end{proof}

\begin{thm}\label{thm:localChar}
    $$\ell(P)=\bigcap\Set{m | m\in\mu S\text{ and } m\subseteq P}.$$
\end{thm}
\begin{proof}
    By \lemref{lem:minimalInt} we know that LHS$\subseteq$RHS.
    
    From \lemref{lem:minimalConv} we know that $p\notin$LHS implies 
    $p\notin$RHS,  i.e. RHS$\subseteq$LHS.
\end{proof}

We can now define the localization of an MV-algebra at a prime 
implication filter. 
\begin{defn}\label{def:localization}
    Let $P$ be a prime implication filter of an MV-algebra $\mathcal 
    L$. Then the \emph{localization of $\mathcal L$ at $P$} is the 
    MV-algebra $\mathcal L/\ell(P)$.
\end{defn}

If $Q\subseteq P$ are two prime implication filters then 
we have $\Set{m | m\in\mu S\text{ and } m\subseteq Q}\subseteq\Set{m | m\in\mu S\text{ and } m\subseteq P}$
and so $\ell(P)\subseteq\ell(Q)$ (from the theorem). Hence there is a natural MV-morphism
$\mathcal L/\ell(P)\to \mathcal L/\ell(Q)$.

And finally a universal property of this localization.

We recall that if $f\colon\mathcal L\to\mathcal M$ is an MV-morphism 
then the \emph{shell} of $f$ is 
$$
\text{sh}(f)=f^{-1}[1]=\Set{x | f(x)=1}
$$
is an implication filter in $\mathcal L$.

\begin{thm}
    Let $P$ be any filter and $f\colon\mathcal L\to\mathcal M$ such 
    that $\text{sh}(f)\subseteq P$ and $\rsf N(\mathcal M)\subseteq 
    f[P]\uparrow$. 
    
    Then $\ell(P)\subseteq\text{sh}(f)$.
\end{thm}
\begin{proof}
    Let $x\notin P$ and $p\in P$. If $f(x)\notin f[P]$ then 
    $f(x)\le f(p)$ and so $f(x\to p)=1$,  i.e. $x\to 
    p\in\text{sh}(f)$.
    
    If $f(x)\in f[P]$ then for some $p\in P$ we have $x\to p$ and 
    $p\to x$ both in the shell of $f$ and hence in $P$. But 
    then $x\meet p= p\otimes (p\to x)\in P$ -- contradiction. 
\end{proof}

From the theorem we see that if $f$ takes $P$ to a filter containing 
all counits then $f$ factorizes through $\mathcal L/\ell(P)$, and so, 
in some sense, $\mathcal L/\ell(P)$ is the largest quotient in which 
$P$ contains all counits (or dominates its complement). 

The assumption that $\text{sh}(f)\subseteq P$ is essential, else the 
theorem yields only that the smaller set 
$\ell(P\join\text{sh}(f))\subseteq\text{sh}(f)$. Indeed if $P,Q$ are 
incomparable prime filters then 
$\rsf N(\mathcal L/Q)=\Set1\subseteq P/Q$ but if $q\in Q\setminus P$ 
and $p\in P\setminus Q$ then $q\to p\in\ell(P)\setminus Q$ -- else 
$p\meet q= q\otimes (q\to p)\in Q$, contradicting $p\notin Q$.

\begin{lem}
    Let $F$ be a prime filter. Then 
    $\ell(P)\subseteq F$ iff $F$ is comparable to $P$.
\end{lem}
\begin{proof}
    If $P\subseteq F$ then $\ell(P)\subseteq P\subseteq F$. If 
    $F\subseteq P$ then $\ell(P)\subseteq \ell(F)\subseteq F$. 
    
    Conversely, if $\ell(P)\subseteq F$ then $F/\ell(P)$ is prime in 
    $\mathcal L/\ell(P)$ and so comparable to $P/\ell(P)$. Hence 
    $F= \eta^{-1}[F/\ell(P)$ is comparable to $\eta^{-1}[P/\ell(P)]= 
    P$.
\end{proof}

\begin{thm}
    $\text{PSpec}(P)$ is order-isomorphic to $\text{PSpec}(\mathcal 
    L/\ell(P))$.
\end{thm}
\begin{proof}
    We know that $\text{PSpec}(\mathcal 
    L/\ell(P))$ is order-isomorphic to \\
    $\Set{ F | F\text{ is a prime 
    filter with }\ell(P)\subseteq F}$ and from the lemma the latter 
    set is $\text{PSpec}(P)$.
%
%
%
%
%
\end{proof}

\end{document}